\documentclass[aps,pra,11pt,tightenlines,longbibliography,notitlepage,amsmath,onecolumn,floatfix,superscriptaddress,eqsecnum]{revtex4-1}
\usepackage[dvipsnames]{xcolor}
\usepackage[colorlinks=true,bookmarks=false,linkcolor=NavyBlue,urlcolor=NavyBlue,citecolor=NavyBlue,breaklinks]{hyperref}
\usepackage{amsmath,amsfonts,amssymb,amsthm}
\usepackage{mathtools}
\usepackage{tabularx}
\usepackage{graphicx}
\usepackage[italicdiff]{physics}
\usepackage{times}
\usepackage{enumitem}
\usepackage{setspace}

\theoremstyle{plain}
\newtheorem{theorem}{Theorem}

\theoremstyle{definition}

\newtheorem{example}{Example}

\theoremstyle{remark}
\newtheorem{remark}{Remark}

\newcommand{\iverson}[1]{[#1]}
\newcommand{\bk}[1]{\qty(#1)}
\newcommand{\Bk}[1]{\qty[#1]}
\newcommand{\BK}[1]{\qty{#1}}

\newcommand{\mc}[1]{\mathcal #1}
\renewcommand{\norm}[1]{\Vert #1 \Vert}

\newcommand{\sset}{\Phi}

\DeclareMathOperator{\expect}{\mathbb E}

\newcommand{\cgraphic}[2]{\centerline{\includegraphics[width=#1\textwidth]{#2}}}
\begin{document}

\title{Ziv-Zakai-type error bounds for general statistical models}

\author{Mankei Tsang}
\email{mankei@nus.edu.sg}
\homepage{https://blog.nus.edu.sg/mankei/}
\affiliation{Department of Electrical and Computer Engineering,
  National University of Singapore, 4 Engineering Drive 3, Singapore
  117583}

\affiliation{Department of Physics, National University of Singapore,
  2 Science Drive 3, Singapore 117551}

\date{\today}


\begin{abstract}
  I propose Ziv-Zakai-type lower bounds on the Bayesian error for
  estimating a parameter $\beta:\Theta \to \mathbb R$ when the
  parameter space $\Theta$ is general and $\beta(\theta)$ need not be
  a linear function of $\theta$.
\end{abstract}

\maketitle

\section{Introduction}
The Ziv-Zakai family of lower bounds on the Bayesian error of
parameter estimation are useful tools in estimation theory, as they
are often reasonable to compute, do not require the involved functions
to be differentiable, and may be much tighter than the
Cram\'er-Rao-type bounds for certain problems
\cite{ziv-zakai,bell97,bell,qzzb,qbzzb,jeong23}. When the underlying
parameter $\theta \in \Theta$ is vectoral in Euclidean parameter space
$\Theta = \mathbb R^p$, Bell and coworkers have proposed some extended
Ziv-Zakai bounds \cite{bell97}, but they work only when the parameter
of interest $\beta(\theta)$ is a linear function of $\theta$.  Jeong,
Dytso, and Cardone recently generalized Bell's bounds for any prior
probability measure on a Euclidean $\Theta$ and
$\beta(\theta) = \theta$ \cite{jeong23}.  This paper generalizes the
bounds for a general parameter space $\Theta$ and a general function
$\beta:\Theta \to \mathbb R$ as the parameter of interest.

\section{Main result}
Let $Y$ be the observation in a Borel space $(\mc Y,\Sigma_{\mc Y})$
and $P_\theta$ be its probability measure conditioned on a hidden
parameter $\theta \in \Theta$. In the Bayesian setting, the parameter
is also a random variable in a Borel space $(\Theta,\Sigma_\Theta)$;
let $\pi$ be its prior probability measure.
For clarity, I use $\theta$ only to denote a specific value in
$\Theta$ and $X$ to denote the hidden parameter as a random variable in
general. Let $\beta:\Theta \to \mathbb R$ be a parameter of interest
and $\check\beta:\mc Y \to \mathbb R$ be an estimator.  Define a
distortion function $D:\mathbb R_+ \to \mathbb R_+$ with
$\mathbb R_+ \equiv \lbrack 0,\infty \rparen$ and the properties
\begin{align}
D(x) &\ge D(y) \quad \textrm{if } x \ge y,
\\
D(0) &= 0,
\\
D(x) &\textrm{ is differentiable, with } \dot D(x) \equiv \dv{D(x)}{x}.
\end{align}
The Bayesian error can then be defined as
\begin{align}
\mathsf E &\equiv \expect\BK{D[|\check\beta(Y)-\beta(X)|]},
\end{align}
where $\expect$ denotes the expectation with respect to all random
variables. For example, $\mathsf E$ is the mean-square error when
$D(x) = x^2$.

Ziv-Zakai-type bounds are lower bounds on $\mathsf E$ in terms of the
minimum error probability of a related binary-hypothesis-testing
problem. Let the two hypotheses be
\begin{align}
\mc H_0: Y &\sim P_\theta, &
\mc H_1: Y &\sim P_{\phi},
\end{align}
and let the prior probability of $\mc H_0$ be $q$.  I denote the
minimum error probability for this hypothesis-testing problem as
$\Pi(\theta,\phi,q)$. I can now present the main result of this paper.

\begin{theorem}
\label{thm}
Let $t \in \mathbb R_+$. For each $t$, let
$\sset_t,\sset_t' \in \Sigma_\Theta$ be two subsets of the parameter
space,
$\Sigma_t \equiv \{S \cap \sset_t: S \in \Sigma_\Theta\}\equiv
\Sigma_\Theta \cap \sset_t$ and
$\Sigma_t' \equiv \Sigma_\Theta \cap \sset_t'$ be the restrictions
\cite[Sec.~1.4]{tao_measure} of $\Sigma_\Theta$ to $\sset_t$ and
$\sset_t'$, respectively, and $T_t:\sset_t \to \sset_t'$ be a
$\Sigma_t/\Sigma_t'$-measurable map such that there exists a measure
$\pi_t$ on $(\sset_t,\Sigma_t)$ satisfying
\begin{align}
\pi &= \pi_t T_t^{-1}
\quad
\textrm{on }(\sset_t',\Sigma_t'),
\label{pit}
\end{align}
and
\begin{align}
\beta(T_t(\theta)) \ge \beta(\theta) + t
\quad
\textrm{almost everywhere on $\sset_t$ with respect to $\pi_t$.}
\label{rule}
\end{align}
Moreover, assume that both $\pi$ and $\pi_t$ are dominated by a
reference measure $\mu$ on $(\sset_t,\Sigma_t)$ and their densities
are denoted as
\begin{align}
f &\equiv \dv{\pi}{\mu},
\\
f_t &\equiv \dv{\pi_t}{\mu}.
\label{ft}
\end{align}
Then
\begin{align}
\mathsf E &\ge  \mathsf Z_1 \equiv 
\frac{1}{2} \int_0^\infty dt \dot D\bk{\frac{t}{2}}
\mc V g_1(t),
\label{bound1}
\\
g_1(t) &\equiv \int_{\sset_t} d\mu(\theta) \Bk{f(\theta)+f_t(\theta)} 
\Pi\qty(\theta,T_t(\theta), \frac{f(\theta)}{f(\theta)+f_t(\theta)}),
\label{g1}
\end{align}
where 
\begin{align}
\mathcal V g(t) &\equiv \sup_{t':t'\ge t} g(t')
\end{align}
is the valley-filling operator \cite{bell}. A more convenient bound is
\begin{align}
\mathsf E &\ge \mathsf Z_2 \equiv
\frac{1}{2} \int_0^\infty dt \dot D\bk{\frac{t}{2}}
\mathcal V g_2(t),
\label{bound2}
\\
g_2(t) &\equiv 
2\int_{\sset_t} d\mu(\theta) \min\BK{f(\theta),f_t(\theta)} \Pi\qty(\theta,T_t(\theta),0.5).
\label{g2}
\end{align}
\end{theorem}
\begin{proof}
With the given properties of $D$, the error can be expressed as
\begin{align}
\mathsf E &= \frac{1}{2} \int_0^\infty dt \dot D\bk{\frac{t}{2}}
\Pr\bk{|\check\beta-\beta| \ge \frac{t}{2}},
\end{align}
where $\Pr$ denotes the probability. Write
\begin{align}
\Pr\bk{|\check\beta-\beta| \ge \frac{t}{2}}
&= \Pr\bk{\check\beta \ge \beta(X) + \frac{t}{2}}
+\Pr\bk{\check\beta < \beta(X) - \frac{t}{2}}.
\label{Pr}
\end{align}
The first probability on the right-hand side of Eq.~(\ref{Pr}) can be bounded as
\begin{align}
\Pr\bk{\check\beta \ge \beta(X) + \frac{t}{2}}
&\ge \int_{\sset_t} d\mu(\theta) f(\theta)
\Pr\bk{\check\beta \ge \beta(\theta) + \frac{t}{2}|X = \theta},
\end{align}
while the second probability can be bounded as 
\begin{align}
\Pr\bk{\check\beta < \beta(X) - \frac{t}{2}}
&\ge \int_{\theta\in \sset_t'} d\pi(\theta) 
\Pr\bk{\check\beta < \beta(\theta) - \frac{t}{2}|X = \theta}
\\
&= \int_{\phi \in \sset_t} d\pi_t(\phi) 
\Pr\bk{\check\beta < \beta(T_t(\phi)) - \frac{t}{2}|X = T_t(\phi)}
\label{measure_change}
\\
&\ge 
\int_{\phi \in \sset_t} d\mu(\phi)
f_t(\phi) \Pr\bk{\check\beta < \beta(\phi) + \frac{t}{2}|X = T_t(\phi)},
\label{magic}
\end{align}
where Eq.~(\ref{measure_change}) comes from Eq.~(\ref{pit}) and the
change-of-variable formula \cite[Theorem~16.13]{billingsley} and
Eq.~(\ref{magic}) comes from Eqs.~(\ref{rule}) and
(\ref{ft}). Equation~(\ref{Pr}) can then be bounded as
\begin{align}
\Pr\bk{|\check\beta-\beta| \ge \frac{t}{2}} &\ge 
g(t) \equiv
\int_{\sset_t} d\mu(\theta) 
\left[f(\theta)\Pr\bk{\check\beta \ge \beta(\theta) + \frac{t}{2}|X = \theta}
\right.
\nonumber\\ &\quad
\left.+ f_t(\theta) \Pr\bk{\check\beta < \beta(\theta) + \frac{t}{2}|X = T_t(\theta)}
\right]
\label{Pr2}
\\
&= 
\int_{\sset_t} d\mu(\theta) [f(\theta)+f_t(\theta)]
\left[\frac{f(\theta)}{f(\theta)+f_t(\theta)}
\Pr\bk{\check\beta \ge \beta(\theta) + \frac{t}{2}|X = \theta}\right.
\nonumber\\ &\quad
\left.+ \frac{f_t(\theta) }{f(\theta)+f_t(\theta)}
\Pr\bk{\check\beta < \beta(\theta) + \frac{t}{2}|X = T_t(\theta)}\right].
\end{align}
The expression in the square brackets can be regarded as the error
probability of a decision rule for a related hypothesis-testing
problem and lower-bounded by $\Pi(\theta,T_t(\theta),f/(f+f_t))$. A
lower bound $g_1(t)$ given by Eq.~(\ref{g1}) on
$\Pr(|\check\beta-\beta|\ge t/2)$ then results.  Since
$\Pr(|\check\beta-\beta|\ge t/2)$ is a nonincreasing function of $t$,
\begin{align}
\Pr(|\check\beta-\beta|\ge\frac{t}{2}) &\ge 
\Pr(|\check\beta-\beta|\ge\frac{t'}{2}) \ge g_1(t'),
\quad
\forall t' \ge t,
\end{align}
the valley-filling operator $\mc V$ can be applied to $g_1$ to obtain
a tighter bound, and Eq.~(\ref{bound1}) results.

Equations~(\ref{bound2}) and (\ref{g2}) can be obtained
similarly by starting with Eq.~(\ref{Pr2}), using the bound
\begin{align}
g(t) &\ge 
2 \int_{\sset_t} d\mu(\theta) \min\BK{f(\theta),f_t(\theta)}
\left[\frac{1}{2} \Pr\bk{\check\beta \ge \beta(\theta) + \frac{t}{2}|X = \theta}
\right.
\nonumber\\&\quad
\left.
+\frac{1}{2} \Pr\bk{\check\beta < \beta(\theta) + \frac{t}{2}|X = T_t(\theta)}\right],
\end{align}
and lower-bounding the expression in the square brackets by
$\Pi(\theta,T_t(\theta),0.5)$.
\end{proof}

\begin{remark}
  At each $t$, $T_t$ and its domain $\sset_t$---where $T_t$ satisfies
  Eq.~(\ref{rule})---may be picked to maximize $g_1(t)$ or $g_2(t)$
  and produce the tightest bounds. If a $T_t$ can be found such that
  Eq.~(\ref{rule}) is an equality, then Eq.~(\ref{magic}) is an
  equality as well, and the bounds may be improved. For a chosen
  $T_t$, its domain $\sset_t$ should also be as big as
  Theorem~\ref{thm} allows.
\end{remark}

\begin{remark}
  Note that $\pi_t$ may not be a probability measure, as the
  pushforward measure $\pi_t T_t^{-1}$ is required to match the given
  prior $\pi$ only on a subset $\sset_t'$ of the parameter space.
\end{remark}

\begin{remark}
  It is straightforward to generalize Theorem~\ref{thm} for quantum
  problems via quantum lower bounds on $\Pi$ \cite{qzzb,qbzzb}.
\end{remark}

\section{Examples}
\begin{example}
  If $\Theta = \mathbb R^p$, $\beta = u^\top\theta$, and
  $T_t(\theta) = \theta + v t$ with $u^\top v = 1$, where $\top$
  denotes the matrix transpose, $\theta$, $u$, and $v$ are all real
  column vectors of dimension $p$, and $u$ and $v$ do not depend on
  $\theta$, then one can pick $\sset_t = \sset_t' = \Theta$,
  $d\mu(\theta) = d^p\theta$, and $f_t(\theta) = f(\theta+vt)$, such
  that Theorem~\ref{thm} becomes the extended Ziv-Zakai bounds in
  Ref.~\cite{bell97}.
\end{example}
\begin{example}
\label{exa_norm}
Let $\Theta = \mathbb R^p$ and the parameter of interest be the
Euclidean norm
\begin{align}
\beta(\theta) &= \norm{\theta} \equiv \sqrt{\theta^\top\theta}.
\end{align}
Suppose that the observation conditioned on each $\theta \in \Theta$
follows the linear Gaussian model $Y \sim N(\theta,\sigma^2 I)$, where
$I$ is the identity matrix.  Previous Ziv-Zakai bounds are
inapplicable because $\beta(\theta)$ here is a nonlinear function of
$\theta$.

It is well known that
\begin{align}
\Pi(\theta,T_t(\theta),0.5)
&= Q\qty(\frac{\norm{\theta-T_t(\theta)}}{2\sigma}),
\end{align}
where $Q(x) \equiv \Pr(Z \ge x)$ for a standard normal $Z$. It is also
straightforward to show that
\begin{align}
T_t(\theta) &= \theta  + \frac{\theta}{\norm{\theta}}t
\label{T_norm}
\end{align}
minimizes $\norm{\theta-T_t(\theta)}$ subject to the constraint
$\beta(T_t(\theta)) = \norm{T_t(\theta)} \ge \beta(\theta) + t =
\norm{\theta}+t$, and the minimum $\norm{\theta-T_t(\theta)}$ is equal
to $t$. Equation~(\ref{bound2}) for the mean-square error becomes
\begin{align}
\mathsf Z_2 &= \int_0^\infty dt t 
\mathcal V\Bk{ A(t) Q\bk{\frac{t}{2\sigma}}},
\\
A(t) &\equiv \int_{\sset_t} d\mu \min\BK{f,f_t}.
\end{align}
In some asymptotic limit of high signal-to-noise ratio (SNR),
$Q(t/2\sigma)$ is expected to be concentrated near $t = 0$ and much
sharper than $A(t)$, so $A(t)$ can be approximated as
$A(t) \approx A(0) = 1$, leading to
\begin{align}
\mathsf Z_2 &\approx \int_0^\infty dt t Q\bk{\frac{t}{2\sigma}}= \sigma^2.
\end{align}
For an example of $A(t)$, consider a uniform prior density with
respect to $d\mu(\theta) = d^p\theta$ for $\norm{\theta} \le R$,
expressed as
\begin{align}
f(\theta) &= \frac{1}{V_p R^p} \iverson{\norm{\theta} \le R},
\end{align}
where $V_p$ is the volume of the unit ball in $p$ dimensions and
$\iverson{\cdot}$ denotes the Iverson bracket. To compute $f_t$, one
can pick $\sset_t = \Theta - \{0\}$ and
$\sset_t' = \{\theta \in \Theta: \norm{\theta} > t\}$ such that $T_t$
is bijective and
\begin{align}
f_t(\theta) &= f(T_t(\theta)) |\det D(\theta)|,
\end{align}
where $D_{jk}(\theta) \equiv \pdv*{[T_t(\theta)]_j}{\theta_k}$ is the
Jacobian matrix \cite[Theorem~17.2]{billingsley}. With some effort, it
can be shown that
\begin{align}
f_t(\theta) &= \frac{1}{V_p R^p} \iverson{\norm{\theta} \le R-t}
\bk{1 + \frac{t}{\norm{\theta}}}^{p-1},
\\
\min\BK{f,f_t} &= \frac{1}{V_p R^p} \iverson{\norm{\theta} \le R-t},
\\
A(t) &= \bk{1 - \frac{t}{R}}^p \iverson{t \le R}.
\end{align}
The width of $A(t)$ is roughly $R/p$, so the high-SNR regime can be
reached if $\sigma \ll R/p$.

$\mathsf Z_2$ can be compared with a Bayesian Cram\'er-Rao bound in a
high-SNR regime \cite{gill95,bcrb_pra} given by
\begin{align}
\mathsf E &\ge \mathsf C \approx 
\expect\bk{u^\top J^{-1} u} = \sigma^2,
&
u_j &\equiv \pdv{\beta}{\theta_j} = \frac{\theta_j}{\beta},
&
J^{-1} = \sigma^2 I,
\end{align}
where $J$ is the Fisher information matrix.
\end{example}
While the two bounds $\mathsf Z_2$ and $\mathsf C$ in
Example~\ref{exa_norm} look identical, note that Cram\'er-Rao-type
bounds in general require $f(\theta)$, $P_\theta$, and $\beta(\theta)$
to be differentiable in some way, whereas Theorem~\ref{thm} here does
not.  More dramatic differences between the two types of bounds may be
obtained when the observation model is nonlinear or non-Gaussian.

Another problem with existing Bayesian Cram\'er-Rao bounds
\cite{bell,gill95,jupp10,bcrb_pra} is that they assume the parameter
space $\Theta$ to be a finite-dimensional differentiable manifold such
that differential geometry can be used, and it is unclear how they may
be generalized for an infinite-dimensional $\Theta$, where
differential geometry becomes a daunting subject. Theorem~\ref{thm},
on the other hand, can deal with an infinite-dimensional $\Theta$, at
least in principle.

\section{Geometric picture}
It may be helpful to visualize $T_t$ in terms of a game in the
parameter space, as shown in Fig.~\ref{game}.  $\beta(\theta)$ sets
the difficulty levels of the game as a function of the player position
$\theta \in \Theta$.  $\{\theta:\beta(\theta) = C\}$ for a constant
$C$ then specifies a surface of constant difficulty, which I call a
stage. Each player starts at a position $\phi \in \sset_t$ and moves to
a final position $T_t(\phi)$ while satisfying certain rules. $t$ is
the transport time; assume that it is given.  Equation~(\ref{rule}) is
a rule that each player must reach a stage with difficulty level
$\beta(T_t(\phi))$ at least $\beta(\phi)+t$, which is proportionate to
the time taken.  Since the stage with the minimum target level
$\{\theta:\beta(\theta) = \beta(\phi)+t\}$ may not exist (i.e., the
set may be empty), the inequality in Eq.~(\ref{rule}) allows the
player to reach a stage that exists at a higher level.

\begin{figure}[htbp!]
\cgraphic{0.7}{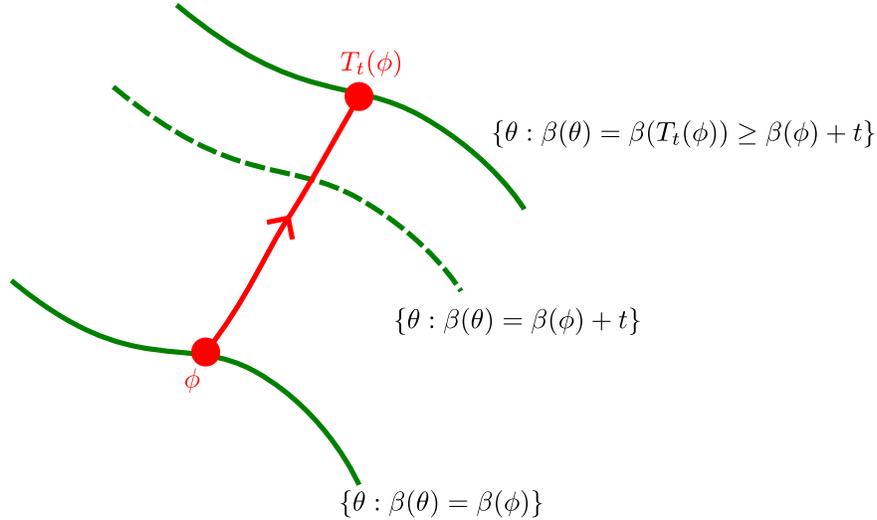}
\caption{\label{game}In the game picture, each player moves from the
  starting position $\phi$ to the final position $T_t(\phi)$ in the
  parameter space. $\beta(\theta)$ can be regarded as the difficulty
  level of a game as a function of player position
  $\theta \in \Theta$. $\beta(\theta) = C$ is then a surface of
  constant difficulty, or a stage. Each player is required to reach a
  stage with difficulty level at least $\beta(\phi) + t$.}
\end{figure}

To understand the role of $T_t$'s domain $\sset_t$, suppose that each
player is uniquely associated with one starting position $\phi$, which
is used as their identification. If a player $\phi$ is unable to
satisfy the rule at time $t$, then they should be removed from the
game by excluding $\phi$ from $\sset_t$. For example, if there is a
supremum difficulty level
$\beta_{\mathrm{sup}} \equiv \sup_{\theta\in\Theta}\beta(\theta) <
\infty$, then a player $\phi$ with
$\beta(\phi) + t > \beta_{\mathrm{sup}}$ can never satisfy
Eq.~(\ref{rule}) because there is no stage at the target level
$\beta(T_t(\phi)) \ge \beta(\phi) + t > \beta_{\mathrm{sup}}$ by time
$t$. The player $\phi$ should hence be excluded from $\sset_t$.

The measure $\pi_t$ assumed by Eq.~(\ref{pit}) can be regarded as the
initial distribution of the players on $\sset_t$.  After the movements
specified by $T_t$, $\pi$ restricted to $\sset_t'$ is their final
distribution.  I call $(\sset_t,T_t,\pi_t)$ a flock. Since $\pi$ is
given, the existence of $\pi_t$ given by Eq.~(\ref{pit}) and its
density $f_t$ given by Eq.~(\ref{ft}) is another rule for the flock.

When there are many qualified flocks, one that produces a high
$g_1(t)$ or $g_2(t)$ should be chosen. $g_j(t)$ can be regarded as a
reward in the game and a definition of flock efficiency.  For example,
Ref.~\cite{bell97} assumes that the stages
$\beta(\theta) = u^\top \theta = C$ are parallel planes in the
Euclidean space, and the flock are a family of parallel displacements
$T_t(\theta) = \theta + v t$ with velocity $v$. If $u^\top v = 1$, the
flock for all starting positions in $\sset_t = \Theta$ are just fast
enough to reach their target stages. Subject to this constraint, $v$
is then chosen to produce a high $g_j(t)$. For another example,
consider the choice of $T_t$ given by Eq.~(\ref{T_norm}) in
Example~\ref{exa_norm}. It minimizes the distance
$\norm{\theta-T_t(\theta)}$, since the $\Pi(\theta,T_t(\theta),0.5)$
there is a decreasing function of the distance and the prior is
ignored. The resulting flock move in radial directions, normal to the
constant-$\beta$ surfaces, as one may heuristically expect from an
efficient flock for the game.

It may be interesting to explore further this geometric picture in
relation to the geometric picture of the Bayesian Cram\'er-Rao bounds
\cite{bcrb_pra}.

This research is supported by the National Research Foundation,
Singapore, under its Quantum Engineering Programme (QEP-P7).

\bibliography{research3}

\begin{thebibliography}{11}%
\makeatletter
\providecommand \@ifxundefined [1]{%
 \@ifx{#1\undefined}
}%
\providecommand \@ifnum [1]{%
 \ifnum #1\expandafter \@firstoftwo
 \else \expandafter \@secondoftwo
 \fi
}%
\providecommand \@ifx [1]{%
 \ifx #1\expandafter \@firstoftwo
 \else \expandafter \@secondoftwo
 \fi
}%
\providecommand \natexlab [1]{#1}%
\providecommand \enquote  [1]{``#1''}%
\providecommand \bibnamefont  [1]{#1}%
\providecommand \bibfnamefont [1]{#1}%
\providecommand \citenamefont [1]{#1}%
\providecommand \href@noop [0]{\@secondoftwo}%
\providecommand \href [0]{\begingroup \@sanitize@url \@href}%
\providecommand \@href[1]{\@@startlink{#1}\@@href}%
\providecommand \@@href[1]{\endgroup#1\@@endlink}%
\providecommand \@sanitize@url [0]{\catcode `\\12\catcode `\$12\catcode
  `\&12\catcode `\#12\catcode `\^12\catcode `\_12\catcode `\%12\relax}%
\providecommand \@@startlink[1]{}%
\providecommand \@@endlink[0]{}%
\providecommand \url  [0]{\begingroup\@sanitize@url \@url }%
\providecommand \@url [1]{\endgroup\@href {#1}{\urlprefix }}%
\providecommand \urlprefix  [0]{URL }%
\providecommand \Eprint [0]{\href }%
\providecommand \doibase [0]{http://dx.doi.org/}%
\providecommand \selectlanguage [0]{\@gobble}%
\providecommand \bibinfo  [0]{\@secondoftwo}%
\providecommand \bibfield  [0]{\@secondoftwo}%
\providecommand \translation [1]{[#1]}%
\providecommand \BibitemOpen [0]{}%
\providecommand \bibitemStop [0]{}%
\providecommand \bibitemNoStop [0]{.\EOS\space}%
\providecommand \EOS [0]{\spacefactor3000\relax}%
\providecommand \BibitemShut  [1]{\csname bibitem#1\endcsname}%
\let\auto@bib@innerbib\@empty
\bibitem [{\citenamefont {Ziv}\ and\ \citenamefont {Zakai}(1969)}]{ziv-zakai}%
  \BibitemOpen
  \bibfield  {author} {\bibinfo {author} {\bibfnamefont {Jacob}\ \bibnamefont
  {Ziv}}\ and\ \bibinfo {author} {\bibfnamefont {Moshe}\ \bibnamefont
  {Zakai}},\ }\bibfield  {title} {\enquote {\bibinfo {title} {Some lower bounds
  on signal parameter estimation},}\ }\href {\doibase 10.1109/TIT.1969.1054301}
  {\bibfield  {journal} {\bibinfo  {journal} {IEEE Transactions on Information
  Theory}\ }\textbf {\bibinfo {volume} {15}},\ \bibinfo {pages} {386–391}
  (\bibinfo {year} {1969})}\BibitemShut {NoStop}%
\bibitem [{\citenamefont {Bell}\ \emph {et~al.}(1997)\citenamefont {Bell},
  \citenamefont {Steinberg}, \citenamefont {Ephraim},\ and\ \citenamefont {{Van
  Trees}}}]{bell97}%
  \BibitemOpen
  \bibfield  {author} {\bibinfo {author} {\bibfnamefont {Kristine~L.}\
  \bibnamefont {Bell}}, \bibinfo {author} {\bibfnamefont {Yossef}\ \bibnamefont
  {Steinberg}}, \bibinfo {author} {\bibfnamefont {Yariv}\ \bibnamefont
  {Ephraim}}, \ and\ \bibinfo {author} {\bibfnamefont {Harry~L.}\ \bibnamefont
  {{Van Trees}}},\ }\bibfield  {title} {\enquote {\bibinfo {title} {Extended
  ziv-zakai lower bound for vector parameter estimation},}\ }\href {\doibase
  10.1109/18.556118} {\bibfield  {journal} {\bibinfo  {journal} {IEEE
  Transactions on Information Theory}\ }\textbf {\bibinfo {volume} {43}},\
  \bibinfo {pages} {624–637} (\bibinfo {year} {1997})}\BibitemShut {NoStop}%
\bibitem [{\citenamefont {{Van Trees}}\ and\ \citenamefont
  {Bell}(2007)}]{bell}%
  \BibitemOpen
  \bibinfo {editor} {\bibfnamefont {Harry~L.}\ \bibnamefont {{Van Trees}}}\
  and\ \bibinfo {editor} {\bibfnamefont {Kristine~L.}\ \bibnamefont {Bell}},\
  eds.,\ \href {\doibase 10.1109/9780470544198} {\emph {\bibinfo {title}
  {Bayesian Bounds for Parameter Estimation and Nonlinear
  Filtering/Tracking}}}\ (\bibinfo  {publisher} {Wiley-IEEE},\ \bibinfo
  {address} {Piscataway},\ \bibinfo {year} {2007})\BibitemShut {NoStop}%
\bibitem [{\citenamefont {Tsang}(2012)}]{qzzb}%
  \BibitemOpen
  \bibfield  {author} {\bibinfo {author} {\bibfnamefont {Mankei}\ \bibnamefont
  {Tsang}},\ }\bibfield  {title} {\enquote {\bibinfo {title} {Ziv-{Zakai}
  {Error} {Bounds} for {Quantum} {Parameter} {Estimation}},}\ }\href {\doibase
  10.1103/PhysRevLett.108.230401} {\bibfield  {journal} {\bibinfo  {journal}
  {Physical Review Letters}\ }\textbf {\bibinfo {volume} {108}},\ \bibinfo
  {pages} {230401} (\bibinfo {year} {2012})}\BibitemShut {NoStop}%
\bibitem [{\citenamefont {Berry}\ \emph {et~al.}(2015)\citenamefont {Berry},
  \citenamefont {Tsang}, \citenamefont {Hall},\ and\ \citenamefont
  {Wiseman}}]{qbzzb}%
  \BibitemOpen
  \bibfield  {author} {\bibinfo {author} {\bibfnamefont {Dominic~W.}\
  \bibnamefont {Berry}}, \bibinfo {author} {\bibfnamefont {Mankei}\
  \bibnamefont {Tsang}}, \bibinfo {author} {\bibfnamefont {Michael J.~W.}\
  \bibnamefont {Hall}}, \ and\ \bibinfo {author} {\bibfnamefont {Howard~M.}\
  \bibnamefont {Wiseman}},\ }\bibfield  {title} {\enquote {\bibinfo {title}
  {Quantum {Bell}-{Ziv}-{Zakai} {Bounds} and {Heisenberg} {Limits} for
  {Waveform} {Estimation}},}\ }\href {\doibase 10.1103/PhysRevX.5.031018}
  {\bibfield  {journal} {\bibinfo  {journal} {Physical Review X}\ }\textbf
  {\bibinfo {volume} {5}},\ \bibinfo {pages} {031018} (\bibinfo {year}
  {2015})}\BibitemShut {NoStop}%
\bibitem [{\citenamefont {Jeong}\ \emph {et~al.}(2023)\citenamefont {Jeong},
  \citenamefont {Dytso},\ and\ \citenamefont {Cardone}}]{jeong23}%
  \BibitemOpen
  \bibfield  {author} {\bibinfo {author} {\bibfnamefont {Minoh}\ \bibnamefont
  {Jeong}}, \bibinfo {author} {\bibfnamefont {Alex}\ \bibnamefont {Dytso}}, \
  and\ \bibinfo {author} {\bibfnamefont {Martina}\ \bibnamefont {Cardone}},\
  }\bibfield  {title} {\enquote {\bibinfo {title} {Functional properties of the
  ziv-zakai bound with arbitrary inputs},}\ }\href {\doibase
  10.48550/arXiv.2305.02970} {\bibfield  {journal} {\bibinfo  {journal} {ArXiv
  e-prints}\ } (\bibinfo {year} {2023}),\ 10.48550/arXiv.2305.02970},\ \Eprint
  {http://arxiv.org/abs/2305.02970} {2305.02970} \BibitemShut {NoStop}%
\bibitem [{\citenamefont {Tao}(2011)}]{tao_measure}%
  \BibitemOpen
  \bibfield  {author} {\bibinfo {author} {\bibfnamefont {Terence}\ \bibnamefont
  {Tao}},\ }\href@noop {} {\emph {\bibinfo {title} {An Introduction to Measure
  Theory}}}\ (\bibinfo  {publisher} {American Mathematical Society},\ \bibinfo
  {address} {Providence},\ \bibinfo {year} {2011})\BibitemShut {NoStop}%
\bibitem [{\citenamefont {Billingsley}(1995)}]{billingsley}%
  \BibitemOpen
  \bibfield  {author} {\bibinfo {author} {\bibfnamefont {Patrick}\ \bibnamefont
  {Billingsley}},\ }\href@noop {} {\emph {\bibinfo {title} {Probability and
  Measure}}},\ \bibinfo {edition} {3rd}\ ed.\ (\bibinfo  {publisher} {Wiley},\
  \bibinfo {address} {New York},\ \bibinfo {year} {1995})\BibitemShut {NoStop}%
\bibitem [{\citenamefont {Gill}\ and\ \citenamefont {Levit}(1995)}]{gill95}%
  \BibitemOpen
  \bibfield  {author} {\bibinfo {author} {\bibfnamefont {Richard~D.}\
  \bibnamefont {Gill}}\ and\ \bibinfo {author} {\bibfnamefont {Boris~Y.}\
  \bibnamefont {Levit}},\ }\bibfield  {title} {\enquote {\bibinfo {title}
  {Applications of the {Van Trees} inequality: A {Bayesian Cramér-Rao}
  bound},}\ }\href {http://www.jstor.org/stable/3318681} {\bibfield  {journal}
  {\bibinfo  {journal} {Bernoulli}\ }\textbf {\bibinfo {volume} {1}},\ \bibinfo
  {pages} {59–79} (\bibinfo {year} {1995})}\BibitemShut {NoStop}%
\bibitem [{\citenamefont {Tsang}(2020)}]{bcrb_pra}%
  \BibitemOpen
  \bibfield  {author} {\bibinfo {author} {\bibfnamefont {Mankei}\ \bibnamefont
  {Tsang}},\ }\bibfield  {title} {\enquote {\bibinfo {title} {Physics-inspired
  forms of the {Bayesian} {Cramér}-{Rao} bound},}\ }\href {\doibase
  10.1103/PhysRevA.102.062217} {\bibfield  {journal} {\bibinfo  {journal}
  {Physical Review A}\ }\textbf {\bibinfo {volume} {102}},\ \bibinfo {pages}
  {062217} (\bibinfo {year} {2020})}\BibitemShut {NoStop}%
\bibitem [{\citenamefont {Jupp}(2010)}]{jupp10}%
  \BibitemOpen
  \bibfield  {author} {\bibinfo {author} {\bibfnamefont {P.~E.}\ \bibnamefont
  {Jupp}},\ }\bibfield  {title} {\enquote {\bibinfo {title} {A van {Trees}
  inequality for estimators on manifolds},}\ }\href {\doibase
  10.1016/j.jmva.2010.03.007} {\bibfield  {journal} {\bibinfo  {journal}
  {Journal of Multivariate Analysis}\ }\textbf {\bibinfo {volume} {101}},\
  \bibinfo {pages} {1814–1825} (\bibinfo {year} {2010})}\BibitemShut
  {NoStop}%
\end{thebibliography}%

\end{document}